\documentclass[a4paper,11pt, leqno]{amsart}
\usepackage{bm}
\usepackage{graphicx}
\usepackage{amssymb}

\usepackage[usenames]{color}

\numberwithin{equation}{section}
\newtheorem{Theorem}{Theorem}[section]
\newtheorem{Thm}[Theorem]{Theorem}
\newtheorem{Fact}[Theorem]{Fact}

\newtheorem{Cor}[Theorem]{Corollary}
 \newtheorem{Lemma}[Theorem]{Lemma}
\newtheorem{Proposition}[Theorem]{Proposition}

\newtheorem{Prop}[Theorem]{Proposition}
\theoremstyle{remark}

\newtheorem{Rmk}[Theorem]{Remark}
\theoremstyle{definition}
\newtheorem{Definition}[Theorem]{Definition}

\newtheorem{Example}[Theorem]{Example}
\newtheorem{Exa}[Theorem]{Example}
\newtheorem*{acknowledgements}{Acknowledgements}
\newcommand{\inner}[2]{\left\langle{#1},{#2}\right\rangle}

\newcommand{\R}{{\mathbb R}}

\newcommand{\C}{{\mathbb C}}
\newcommand{\E}{{\mathbb E}}

\newcommand{\mc}[1]{{\mathcal #1}}

\newcommand{\pmt}[1]{{\begin{pmatrix} #1  \end{pmatrix}}}

\renewcommand{\phi}{\varphi}
\renewcommand{\epsilon}{\varepsilon}
\newcommand{\op}[1]{{\operatorname{ #1}}}

\renewcommand{\phi}{\varphi}
\renewcommand{\epsilon}{\varepsilon}

\renewcommand{\phi}{\varphi}

\newcommand{\mathsym}[1]{{}}
\newcommand{\unicode}[1]{{}}

\title{
Hopf differentials and curvature line flows on time-like CMC surfaces
}

\author{Naoya Ando}
\address[Naoya Ando]{%
Department of Mathematics,
Faculty of Advanced Science and Technology, 
Kumamoto University, 2-39-1 Kurokami, Chuo-ku,  Kumamoto 860-8555, Japan
}
\email{andonaoya@kumamoto-u.ac.jp}

\author{Masaaki Umehara}
\address[Masaaki Umehara]{%
   Department of Mathematical and Computing Sciences,
  Institute of Science Tokyo,
   2-12-1-W8-34, O-okayama, Meguro-ku,
   Tokyo 152-8552, Japan.
}
\email{umehara@comp.isct.ac.jp}


\begin{document}

\maketitle
\begin{abstract}
We investigate the relationship between the Hopf differentials
and the curvature line flows on time-like constant mean curvature (CMC)
surfaces in Lorentzian 3-space forms.
In particular, when the Hopf differential is non-degenerate,
the index of a curvature line flow at an umbilic point depends
precisely on the remainder of its order modulo four.
\end{abstract}

\section*{Introduction}
We denote by $M^3_1(c)$ 
the (Lorentzian) {\it 3-dimensional space form}
of signature $(++-)$ with 
constant sectional curvature $c\in \R$
(for details, refer to the beginning of Section~2).
An immersion $f:U\to M^3_1(c)$ 
defined on a domain $U$ of the $uv$-plane $\R^2$
is called a {\it regular surface}.
A point $p\in U$ is
called an {\it umbilic} of $f$
if the shape operator at $p$ is a scalar multiple of
the identity transformation. Similarly,
a point $p\in U$ is
called a {\it quasi-umbilic} of $f$
if $p$ is not an umbilic 
but the two principal curvatures coincide at $p$.

If $f$ is space-like,
then quasi-umbilics never
appear, and the principal 
curvatures are real-valued, but this 
may not always be true if $f$ is time-like.
We now assume that $f$ is a time-like regular surface
and fix an umbilic or quasi-umbilic $p\in U$.

\begin{Definition}\label{def:0}
The point $p$ is said to be \emph{admissible} if there exists a neighborhood 
$U_1 \subset U$ of $p$ on which the principal curvatures of $f$ are real-valued. 
If $p$ is an umbilic of $f$, 
then it is said to be \emph{isolated} if it is 
admissible 
and the above neighborhood 
$U_1$ can be chosen so that no other umbilics occur in $U_1$.
\end{Definition}

In the present paper, 
we investigate 
umbilics and quasi-umbilics 
on a time-like CMC surface $f$ in $M^3_1(c)$,  
where a regular surface with constant mean curvature  
will be referred to as a ``CMC surface'' for short.  
The properties of umbilics and quasi-umbilics of $f$
are determined by
its Hopf differential $Q_f$, as in the space-like case.
However, unlike the space-like or Euclidean cases,
several remarkable phenomena occur in the curvature line flows 
of time-like CMC surfaces.
In fact, we will show the following:
\begin{itemize}
\item[(I)]
{\it Any isolated umbilic on a time-like CMC surface $f$
is an accumulation point of the set of quasi-umbilics}.
When a given time-like surface is not of constant 
mean curvature, 
it might admit umbilics that do not become
accumulation points of quasi-umbilics (cf.~\cite{AU}).

\item[(II)]
{\it Suppose that the Hopf differential $Q_f$
of a time-like CMC surface $f$ 
is of finite type $($cf.~Definition~\ref{def:FT}$)$ 
and is non-degenerate at an umbilic $p$ 
$($cf.~Definition~\ref{def:QS}$)$.
Then the order $m(\ge 1)$ of $Q_f$ at $p$ is well defined, 
and the behavior of the curvature line flows of $f$
around $p$ depends on the remainder of $m$ modulo~$4$. 
More precisely, if
$p$ is an isolated umbilic, then
$m$ is even.
Moreover, in this case,  
there exists a smooth curvature line flow of $f$  
with index $1$ or $-1$ at $p$  
if and only if $m$ is even but not divisible by~$4$.}
\end{itemize}
This dependence on the remainder of $m$ modulo~4
suggests a subtle geometric structure
inherent in time-like CMC surfaces.
To contrast the time-like case with the space-like case, 
umbilics of space-like 
CMC surfaces are also investigated in Appendix~A.

The paper is organized as follows:
In Section~1,  
we recall basic facts on time-like regular surfaces.
In Section~2, we define
the Hopf differentials of time-like CMC surfaces and explain their
properties. In Section~3, we study the behavior of 
the curvature line flows of
time-like CMC surfaces in 
neighborhoods of umbilics and quasi-umbilics.
We prepare two appendices: one concerning 
umbilics of space-like CMC surfaces,
and the other containing 
some background on paracomplex geometry.

\section{Preliminaries}

Let $M^3_1$ be a Lorentzian $3$-manifold.  
We fix a domain $U$ in the $uv$-plane $\R^2$, and
consider a time-like surface $f:U\to M^3_1$.
By the well-known existence theorem 
of isothermal coordinates (cf.~\cite{W}),
the first fundamental form of $f$ can be written as
\begin{equation}\label{283}
ds^2=e^{2\sigma}(du^2-dv^2),
\end{equation}
where $\sigma$ is a $C^\infty$ function on $U$.
The $uv$-plane can be identified with the set of paracomplex numbers:
$
\check{\C}:=\{z=u+jv\,;\, u,v\in \R\},
$
where $j$ is the para-imaginary unit (i.e., $j^2=1$). 
The notation 
\begin{equation}\label{eq:N20}
N^2[a+jb]:=a^2-b^2 \qquad (a,b\in \R)
\end{equation}
for a paracomplex number can be regarded as
an analogue of the square of the norm of a complex number, 
and it satisfies
$$
N^2[(a+jb)(c+jd)]=N^2[a+jb]\, N^2[c+jd] \qquad (a,b,c,d\in \R).
$$

Let
$\nu$ be a unit normal vector field of $f$.
The second fundamental form $h:=Ldu^2+2Mdudv+Ndv^2$
is given by
$$
L:=\inner{\nabla_uf_{u}}{\nu},\quad M:=\inner{\nabla_u f_{v}}{\nu},\quad
N:=\inner{\nabla_vf_{v}}{\nu},
$$
where \lq\lq$\inner{\,\,}{\,\,}$" is the Lorentzian metric on $M^3_1$,
$$
f_u:=df(\partial/\partial u),\quad 
f_v:=df(\partial/\partial v),\quad
\nabla_u:=\nabla_{\partial/\partial u},\quad
\nabla_v:=\nabla_{\partial/\partial v},
$$
and $\nabla$ 
is the Levi-Civita connection of $M^3_1$.
The matrix defined by
\begin{equation}\label{eq:W259}
W_f:=
e^{-2\sigma}
\pmt{L & M \\
     -M & -N}
\end{equation}
is called the {\it Weingarten matrix} of $f$,
whose eigenvalues  
$\lambda_1,\lambda_2$ at each point of $U$	
are the {\it principal curvatures} of $f$.
We set 
\begin{equation}\label{eq:274}
D_f:=\op{trace}(W_f)^2-4\det(W_f)
=e^{-4\sigma}
\Big(
(L+N)^2-4M^2
\Big).
\end{equation}

\begin{Rmk}
The trace-free part $A_f$ of $e^{2\sigma}W_f$ can be written as
\begin{equation}\label{eq:Af}
A_f=
\frac12
\pmt{
L+N & 2M \\
-2M & -(L+N)
}.
\end{equation}
It satisfies
\begin{equation}\label{eq:Af2}
D_f=-4 e^{-4\sigma}\det(A_f).
\end{equation}
Hence, the integral curves of an eigenvector field of $A_f$ are
the curvature lines of $f$.
Later, we will introduce 
$Q_f:=\big((L+N)+2jM\big)dz^2$ 
as the Hopf differential
of $f$. 
Then $A_f$ corresponds to the matrix of coefficients of $Q_f$.
\end{Rmk}

\begin{Proposition}\label{thm:633A}
Let $p\in U$ be an isolated umbilic of a time-like surface
$f:U\to M^3_1$, and let
$\mc F$ be a $C^r$ curvature-line flow 
on $U\setminus \{p\}$, where $r\ge 1$.
Then there exists a unique $C^r$ curvature-line flow 
$\mc F^\perp$ on $U\setminus \{p\}$
such that
\begin{enumerate}
\item each leaf of $\mc F^\perp$ is perpendicular to 
those of $\mc F$,
\item the indices $i_{\mc F}$ of $\mc F$ and
$i_{\mc F^\perp}$ of $\mc F^\perp$ at $p$ satisfy
$i_{\mc F^\perp}=-i_{\mc F}$.
\end{enumerate}
\end{Proposition}

\begin{proof}
This proposition is stated in the setting 
of an isothermal coordinate system.
In the case of the Ribaucour parameter, the same result was shown
by Ando and Umehara~\cite[Proposition~C]{AU}.
The proof is almost parallel to theirs.
For the readers' convenience, we give a proof here:
We may assume that 
the first fundamental form of $f$ is given as in
\eqref{283}.
We fix $q\in U\setminus \{p\}$ arbitrarily.
Then there exist a neighborhood $V(\subset U\setminus \{p\})$ of 
$q$ and a vector field $X_1$ on $V$
whose integral curves correspond to the leaves of $\mc F$. 
We can write
\[
X=x_1 \frac{\partial}{\partial u}+x_2 \frac{\partial}{\partial v}
\]
on $V$. Since $(u,v)$ is isothermal,
\[
X^\perp=x_2 \frac{\partial}{\partial u}+x_1 \frac{\partial}{\partial v}
\]
also gives a vector field on $V$
which points in an eigen-direction of $W_f$ at each point of $V$,
and therefore the integral curves of $X^\perp$ define 
a curvature-line flow 
$\mc F^\perp$ which is perpendicular to $\mc F$.
Since $q\in U\setminus \{p\}$ is chosen arbitrarily, this curvature-line flow
$\mc F^\perp$ can be extended uniquely on $U\setminus \{p\}$.
Since $(x_1,x_2)\mapsto (x_2,x_1)$ is an orientation-reversing
diffeomorphism,  we have $i_{\mc F^\perp}=-i_{\mc F}$.
\end{proof}

We recall the following:

\begin{Fact}[{\cite[Remark 4.1]{AU}}]\label{fact:329}
Let $a,b,c$ be real numbers, and consider 
two matrices
\[
A:=\pmt{a & c \\ -c & -b},\qquad
B:=\pmt{c & \frac{a+b}{2} \\[2pt] \frac{a+b}{2} & c}.
\]
Then the discriminant $D_A$ of the
eigenequation for the
matrix $A$ satisfies
$D_A=-4\det(B)$.
In this setting,
$A$ is a scalar multiple of
the identity matrix if and only if $B=0$.
Moreover, 
for a given nonzero column vector $\bm v:={}^t(x,y)$ 
$($the transpose of the row vector $(x,y))$, 
the following two properties are equivalent:
\begin{enumerate}
\item $\bm v$ is an eigenvector of $A$;
\item $\bm v$ is a {\it null vector} of $B$, that is,
${}^t{\bm v} B \bm v=0$.
\end{enumerate}
Furthermore, if $\bm v(=x+jy)$ satisfies (1) and (2),
then $N^2(\bm v)=0$ holds if and only if $D_A=0$.
\end{Fact}

\begin{proof}
Here we only prove the last statement.
If $D_A=0$, then $|a+b|/2=|c|$.
Then $B$ is a scalar multiple of
\[
\pmt{1 & 1 \\ 1 & 1} \quad \text{or} \quad 
\pmt{1 & -1 \\ -1 & 1}.
\]
Hence ${}^t{\bm v} B \bm v=0$ implies that $x^2-y^2=0$.
The above argument also shows that the converse holds.
\end{proof}

\section{The Hopf differentials of time-like CMC surfaces}

We denote by $M^3_1(c)$ ($c\in \R$)
the (Lorentzian) {\it 3-dimensional space form}, which is
the Lorentzian 3-manifold of 
constant sectional curvature $c$, as follows:
\begin{itemize}
\item the Lorentz--Minkowski 3-space $\R^3_1$ 
of signature $(++-)$ if $c=0$,
\item the de Sitter 3-space in $\R^4_1$ if $c>0$,
\item the anti-de Sitter 3-space in $\R^4_2$ if $c<0$,
\end{itemize}
where $\R^4_1$ (resp.\ $\R^4_2$) is the affine 4-space with the canonical metric 
of signature $(+++\,-)$ (resp.\ $(++-\,-)$).
In this section we consider a time-like surface $f:(U;u,v)\to M^3_1(c)$
such that $(u,v)$ form an isothermal coordinate system
as in \eqref{283}.

\subsection*{Para-holomorphicity of Hopf differentials}

The notion of para-holomorphicity is the Lorentzian counterpart 
of holomorphicity in the Euclidean case. 
In this sense, the structure of time-like 
CMC surfaces can be viewed as the \lq\lq split-complex" analogue 
of the classical theory of minimal and CMC surfaces in $\E^3$.
We first set $c=0$, that is, $M^3_1(c)=\R^3_1$.
For a time-like surface $f:U\to \R^3_1$
defined on an isothermal coordinate neighborhood in the
$uv$-plane, we have the following Frenet-type equations:
\[
\bm F_u=\bm F\Omega,\qquad \bm F_v=\bm F\Lambda,
\]
where $\bm F=(f_u,f_v,\nu),$ and
\[
\Omega=\pmt{
\sigma_u & \sigma_v & -e^{-2\sigma}L \\
\sigma_v & \sigma_u & e^{-2\sigma}M \\
L & M & 0  
},
\qquad
\Lambda=\pmt{
\sigma_v & \sigma_u & -e^{-2\sigma}M \\
\sigma_u & \sigma_v & e^{-2\sigma}N \\
M & N & 0  
}.
\]
Here $\sigma$ is a smooth function as in
\eqref{283}.
These equations yield the integrability 
condition
\[
\Lambda \Omega+\Omega_v=\Omega\Lambda+\Lambda_u
\]
on $U$.
Computing this, we obtain the following three relations:
\begin{align}
LN-M^2&=
e^{2\sigma}(\sigma_{uu}-\sigma_{vv}), \\[2pt]
L_v-M_u&=\sigma_v (L-N), \label{C1} \qquad
N_u-M_v=\sigma_u (-L+N).
\end{align}
The first one
is the Gauss equation,
and the others are the Codazzi equations.
Similarly, if $f$ is a map into the de Sitter 3-space in $\R^4_1$
or the anti-de Sitter 3-space in $\R^4_2$, then 
we can show that the Codazzi equations are
written as in \eqref{C1}.
We now assume that the mean curvature function
\[
H=e^{-2\sigma}(L-N)/2
\]
is constant.
Then we have
\begin{align}\label{1}
0&=H_u=-2\sigma_u e^{-2\sigma}(L-N)+e^{-2\sigma}(L-N)_u, \\[2pt]
\label{2}
0&=H_v=-2\sigma_v e^{-2\sigma}(L-N)+e^{-2\sigma}(L-N)_v.
\end{align}

\begin{Prop}
A point $p\in U$  is an umbilic of $f$ if and only if
\begin{equation}\label{eq:491}
\hat Q_f:=L+N+2jM
\end{equation}
vanishes at $p$. 
If $f$ has constant mean curvature, 
then $\hat Q_f$ is para-holomorphic
with respect to the paracomplex coordinate system
$z:=u+jv$. Equivalently, $\hat Q_f$ satisfies 
the following para-Cauchy--Riemann equations $($cf.\ Appendix~B$)$:
\begin{equation}\label{eq:501}
(L+N)_u=2M_v,\qquad (L+N)_v=2M_u.
\end{equation}
\end{Prop}

\begin{proof}
A point $p\in U$  is an umbilic of $f$ if and only if
$(L+N)/2=M=0$, which follows from 
Fact~\ref{fact:329}.
By \eqref{C1},
\eqref{1} and \eqref{2}, we obtain \eqref{eq:501}.
\end{proof}

\begin{Rmk}
By \eqref{eq:274} and
\eqref{eq:491}, we have
\begin{equation}\label{eq:27}
N^2(\hat Q_f)=(L+N)^2-4M^2=e^{4\sigma}D_f.
\end{equation}
\end{Rmk}

\begin{Definition}
We call $Q_f:=\hat Q_f\,dz^2$ the {\it Hopf differential} of $f$,
which does not depend on the choice of the paracomplex coordinate system $z$.
\end{Definition}

For the sake of simplicity, we set $o:= (0,0)\in \check \C$
and assume $p:=o$.
By Proposition~\ref{prop:c},
we can write
\begin{align}\label{eq:500}
\hat Q_f(z)&=\epsilon_1 \phi_1(x)+
\epsilon_{-1} \phi_{-1}(y), \\[2pt]
\label{eq:950}
N^2(\hat Q_f(z))&=\phi_1(x)\phi_{-1}(y),
\end{align}
where $x:= (u+v)/2\,\,\, y:= (u-v)/2$.
Since $dz=du+jdv=2\epsilon_1dx+2\epsilon_{-1}dy$,
we have
$dz^2=4\epsilon_1dx^2+4\epsilon_{-1} dy^2$ 
by \eqref{eq:eee} in the appendix.
Hence, we have
\begin{equation}\label{eq:931}
Q_f(z)=
4\epsilon_1 \phi_1(x)\,dx^2
+4\epsilon_{-1} \phi_{-1}(y)\,dy^2.
\end{equation}

\begin{Definition}\label{def:FT}
We say that $Q_f$ is of {\it finite type} at $p$
if,
for each $s\in \{1,-1\}$, either 
\begin{itemize}
\item[(1)] $\phi_s$ vanishes identically as a function germ, or
\item[(2)] 
there exists a non-negative integer $m_s$
such that
$$
c_s:=\frac1{m_s!}\frac{d^{m_s} \phi_s(0)}{dt^{m_s}}
$$
is the coefficient of the first non-vanishing jet of $\phi_s$.
\end{itemize}
In this setting, if (1) holds, we set $m_s:=\infty$
and $c_s:=0$.
Then 
we call 
$$
m_1,m_{-1}\in \{0,1,2,\ldots\}\cup \{\infty\}
$$
the {\it split-orders} of $Q_f$ at $p$.
\end{Definition}

We can construct examples of time-like zero mean curvature
surfaces (i.e.\ ZMC-surfaces)
in $\R^3_1$ using the Weierstrass--type representation formula 
given by Konderak~\cite{KO}:
Let $g(z)$ be a para-holomorphic function
satisfying $g(o)=0$ and $\omega(=\hat \omega(z)\,dz)$ a 
para-holomorphic $1$-form satisfying $N^2(\hat \omega(o))\ne 0$.
We set $\op{Re}(z):=u$ and $\op{Im}(z):=v$.
Then 
\begin{equation}\label{eq:KO}
f(u,v)=\op{Re}\int_o^z (j(1-g^2),\,2g,\,1+g^2)\,\omega
\end{equation}
gives a time-like ZMC-surface. 
This formula
is essentially the same as Konderak's one.
In fact, after changing the coordinates $(t,x,y)$ to $(y,x,-t)/2$
and replacing $(g,\omega)$ by $(jg,j\omega)$,
this formula coincides with the one in~\cite[(3.14)]{KO}.
The first fundamental form of $f$ is
$$
ds^2=(1+N^2[g])^2 N^2[\hat \omega]\,(du^2-dv^2),
$$
and the Hopf differential is
\begin{equation}\label{eq:H666}
Q_f=-\omega\,dg.
\end{equation}
We give here an example 
whose  Hopf differential is not of 
finite type.

\begin{Exa}\label{exA2}
We set 
$
\phi(t):=\exp(-1/t^2)
$
($t\in \R$),
which can be considered as a $C^\infty$ function on $\R$
such that all the derivatives $d^k\phi/dt^k$ ($k=0,1,2,\ldots$)
at $t=0$ vanish.
Then (cf.\ \eqref{X}),
$\hat \omega(z):=\phi\vee \phi(u,v)$
is a para-holomorphic function on $\check \C$.
Substituting the data $(g,\omega):=(z,\hat \omega(z)\,dz)$ into
\eqref{eq:KO}, we obtain a time-like ZMC-surface $f$ whose Hopf
differential $Q_{f}:=-\hat \omega(z)\,dz^2$ (cf.\ \eqref{eq:H666})
is not of finite type.
\end{Exa}

\begin{Prop}
Let $f:U\to \R^3_1$ be
 a time-like CMC surface which is not totally umbilic.
Suppose that $p\in U$ is 
an umbilic
such that
$Q_f(z)$ is of finite type at $p$.
Then there exists an isothermal coordinate system $(u,v)$ centered at $p$
such that
\begin{equation}\label{eq:777}
\hat Q_f(z)=\epsilon_1 x^{m_1}\psi_1(x)+
\epsilon_{-1} y^{m_{-1}}\psi_{-1}(y)
\end{equation}
if $m_1$ and $m_{-1}$ are finite, and
\begin{equation}\label{eq:Q672a}
\hat Q_f(z)=\epsilon_1 x^{m_1}\psi_1(x)
\qquad\Big(\text{resp.\,\,}
\hat Q_f(z)=\epsilon_{-1} y^{m_{-1}}\psi_{-1}(y)\Big)
\end{equation}
if $m_{-1}=\infty$ $($resp.\ $m_{1}=\infty)$,
where $\psi_{1}$ and $\psi_{-1}$ are smooth functions of one variable
satisfying $\psi_1(0)\ne 0$ and $\psi_{-1}(0)\ne 0$.

Moreover, if $m_1$ and $m_{-1}$ are finite, the split-orders
$m_1,m_{-1}$ and the sign of
$\psi_{1}(0)\psi_{-1}(0)$ do not depend on the choice of
the isothermal coordinate system.
\end{Prop}

\begin{proof}
Here we consider the case where
$m_{1},\,\,m_{-1}<\infty$. 
We can write
\begin{equation}\label{eq:618}
\phi_{1}(x)=x^{m_1}\psi_{1}(x),\quad \phi_{-1}(y)=y^{m_{-1}}\psi_{-1}(y)
\qquad (\psi_{1}(0)\psi_{-1}(0)\ne 0).
\end{equation}
From \eqref{eq:500}, we obtain the expression \eqref{eq:777}.

Let $(a,b)$ be another orientation-compatible
isothermal coordinate system of $f$
centered at $p$.
Then $w:=a(u,v)+jb(u,v)$ is a para-holomorphic function,
and we define $\hat Q_1$ so that
$\hat Q_1dw^2=\hat Q_fdz^2$ holds.
Thus, we have
$\hat Q_1 =\hat Q_f(dz/dw)^2$
and
$$
\frac{dz}{dw}=\epsilon_1 \theta_1(s) +\epsilon_{-1} \theta_{-1}(t)
\qquad (s:=\tfrac{a+b}{2},\quad t:=\tfrac{a-b}{2}),
$$
where $\theta_1(0)\theta_{-1}(0)\ne 0$.
By this expression
with
\eqref{eq:eee},
we have
$$
\hat Q_1=\epsilon_1 x^{m_1}\psi_{1}(x)\theta_1(s)^2 +
\epsilon_{-1}y^{m_{-1}}\psi_{-1}(y)\theta_{-1}(t)^2.
$$
In this expression we can write
$x=x(s)$ and $y=y(t)$ 
so that $x(0)=y(0)=0$,
$dx(0)/ds\ne 0$ and $dy(0)/dt\ne 0$.
Hence,  $m_1$, $m_{-1}$ and the sign of $\psi_1(0)\psi_{-1}(0)$
are invariant under isothermal coordinate changes.
\end{proof}

\begin{Definition}\label{def:QS}
In the expression \eqref{eq:777} at $p$,
if  $(m:=)m_1=m_{-1}$ holds,
we say that $Q_f$ is
{\it non-degenerate} at $p$, otherwise we say that $Q_f$ is {\it degenerate} at $p$.
If $Q_f$ is non-degenerate at $p$,
then $m$ is called the {\it order} of $Q_f$ at $p$.
\end{Definition}

When $Q_f$ is
non-degenerate at $p$, we can prove the following:

\begin{Cor}\label{prop:488}
If $Q_f$ is
non-degenerate at $p$,
then
$$
m=\op{Min}\left\{i\ge 0\,;\, \frac{d^i\hat Q_f(p)}{dz^i}\ne 0\right\}
$$
holds. 	
Moreover, 
we can write
\begin{equation}\label{eq:614}
\hat Q_f(z)= R(z)z^m,
\qquad 
N^2(R(0))=2^{-2m}\psi_1(0)\psi_{-1}(0)\ne 0,
\end{equation}
where $R(z):=2^{-m}(\epsilon_1\psi_1(x)+\epsilon_{-1} \psi_{-1}(y))$.
\end{Cor}

\begin{Rmk}\label{rem:524}
When $Q_f$ vanishes identically, $f$ is totally umbilic, and such a case 
will not be treated in this paper, since it is an exceptional one. 
For example, a totally umbilic time-like surface in $\R^3_1$ 
is congruent to a plane or a part of the de Sitter plane of constant curvature 
$c(>0)$,
which can be proved by imitating the Euclidean case.
\end{Rmk}

\section{Curvature line flows around umbilics}
Throughout this section, we fix a time-like CMC surface
$
f:U\to M^3_1
$
such that the first fundamental 
form $ds^2$ satisfies \eqref{283},
and the Weingarten matrix of $f$ is given by \eqref{eq:W259}.
We fix a point $p\in U$ arbitrarily and assume that
$Q_f(z)$ is of finite type at $p$. 
Without loss of generality, we may assume that $p:=o$.
By \eqref{eq:27}, \eqref{eq:950} and \eqref{eq:618},
we can write
\begin{align}\label{eq:3-2}
&e^{4\sigma}D_f=
N^2(\hat Q_f(z))=x^{m_1}y^{m_{-1}}\psi_1(x)\psi_{-1}(y)
\end{align}
with the point $p$ as the base point,
if $m_1$ and $m_{-1}$ are not equal to $\infty$.
Set
\begin{align*}
&\mc L_1:=\{u+jv \in \check \C\,;\, v=u\}, \qquad
\mc L_{-1}:=\{u+jv \in \check \C\,;\, v=-u\}.
\end{align*}
By definition, we have the following equivalence:
\begin{equation}\label{eq:z617}
z(=(u+jv))\in \mc L_s \,\, \Longleftrightarrow\,\, u-sv=0.
\end{equation}
Then $\mc L_1\cup \mc L_{-1}$ can be considered as the set of
non-invertible elements of $\check \C$.

\begin{Lemma}\label{lem:614}
The following two conditions are equivalent:
\begin{enumerate}
\item $p$ is a quasi-umbilic,
\item 
there exists $s\in \{1,-1\}$ such that
$1\le m_s<\infty$ and $m_{-s}=0$.
\end{enumerate}
If one of these occurs, then
$U\cap \mc L_{-s}$ is the
set of quasi-umbilics of $f$ along which 
the $($unique$)$ principal direction of $f$ 
is parallel to the line $\mc L_{s}$.
\end{Lemma}

\begin{proof}
It is obvious that
$p$ is a quasi-umbilic if and only if
there exists $s\in \{1,-1\}$ such that
$1\le m_s<\infty$ and $m_{-s}=0$.
In this setting, 
as in the statement of (2) in Lemma \ref{lem:614}, 
$U\cap \mc L_{-s}$ coincides with 
the
set of quasi-umbilics of $f$.
Now fix  $q\in \mc L_{-s}$ arbitrarily.
Then $a+sb=0$, where $a:=L(q)+N(q)$ and $b:=2M(q)$.
Hence the matrix $A_f(q)$ (given in \eqref{eq:Af})
has a unique $1$-dimensional eigenspace  
generated by  $(s,1)\in \mc L_{s}$, 
proving the conclusion.
\end{proof}

\begin{Definition}\label{def:260}
A point $p$ is called a {\it positive point}
$($resp.\ {\it negative point}$)$
if  $D_f(p)>0$ (resp.\ $D_f(p)<0$). 
Then  the principal curvatures of $f$
at 
a  positive $($resp.\ negative$)$ point
are $($resp.\ are not$)$ real numbers. 
\end{Definition}

Let $V(\subset U)$ be a connected neighborhood  of $o$.
We 
denote by
$\mc P(V)$ (resp.\ $\mc N(V)$) the
set of positive (resp.\ negative)
points on $V$
in the sense of Definition~\ref{def:260},
and denote by $\Xi(V)$ (resp.\ $\Xi'(V)$)
the set of umbilics (resp.\ quasi-umbilics). 
We set 
$$
\mc Z(V):=\Xi(V)\cup \Xi'(V).
$$
By definition, $\mc Z(V)$ coincides with
the zero set of $D_f$,
so we have the following disjoint union:
$$
V=\mc P(V)\cup \mc N(V)\cup \mc Z(V).
$$

\subsection*{The case $m_1=\infty$ or $m_{-1}=\infty$}

In this subsection, we consider the case where
$m_1=\infty$ or $m_{-1}=\infty$:

\begin{Prop}
Let $f:U\to M^3_1(c)$ be a time-like CMC immersion.
Suppose that 
$(0\le) m_s\le m_{-s}=\infty$.
Then
$\mc Z(V)=V$ holds.
Moreover, if $f$ is a real analytic map,
then one of the following three cases occurs:
\begin{enumerate}
\item 
$m_1=m_{-1}=\infty$, that is, $f$ is a totally umbilic
surface;
\item $m_s=0$ and $m_{-s}=\infty$,
that is, $f$ is a
totally quasi-umbilic
surface, i.e., $\Xi'(V)=V$;
\item $1\le m_s<\infty$ and $m_{-s}=\infty$.
In this case, we have 
$$
\Xi(V)=\mc L_{-s}\cap V,\quad
\Xi'(V)=V\setminus \mc L_{-s}.
$$
\end{enumerate}
\end{Prop}

\begin{proof}
If $m_1=m_{-1}=\infty$, then $Q_f$ vanishes identically, 
since $\hat Q_f$ is
a real analytic function; thus (1) follows.
We then consider the case 
where $1\le m_s<\infty$ and $m_{-s}=\infty$.
Then we have (cf.\ \eqref{eq:Q672a})
\begin{align*}
&\hat Q_f(z)=\epsilon_s (u+sv)^{m_s}\psi_s(u+sv),
\qquad
N^2(\hat Q_f(z))=0.
\end{align*}
So $\hat Q_f(z)$ vanishes if and only if $u+s v=0$,
which implies $z\in \mc L_{-s}$.
Hence $\Xi(V)=\mc L_{-s}\cap V$ and
$\Xi'(V)=V\setminus \mc L_{-s}$ hold.
Similarly, when $m_s=0$ and $m_{-s}=\infty$,
we obtain
$\Xi(V)=\varnothing$ and $\Xi'(V)=V$.
\end{proof}

\begin{Exa}\label{exA1}
We set $(g,\omega):=(-1,dz)$. Then $Q_f$ is identically zero,
and the corresponding surface is a plane. In this case,
$m_1=m_{-1}=\infty$.
On the other hand, if we set
$g:=-2\epsilon_1z$ and $\omega:=dz$,
then 
$$
Q_f=2\epsilon_1dz^2=8\epsilon_1dx^2,
$$
and so $m_1=0$ and $m_{-1}=\infty$.
The resulting ZMC-surface is given by
$$
f(x,y):=
\left(x-\frac{16x^3}{3},-4x^2,
x+\frac{16x^3}3\right)+y(-1,0,1),
$$
which is a totally quasi-umbilic ruled surface defined on  $\R^2$.  
As shown by Clelland~\cite{C}, any totally quasi-umbilic 
surface is a ruled surface.
\end{Exa}

\subsection*{Properties of quasi-umbilics}

From now on, we assume that 
$m_1$ and $m_{-1}$ are
finite.
In this subsection, we assume that $o(:=p)$ is a quasi-umbilic.

\begin{Thm}\label{thm:909}
Let $f:U\to M^3_1(c)$ be a time-like CMC immersion.
Suppose that 
$m_1,m_{-1}<\infty$.
If $o$ is a quasi-umbilic, then 
there exists a neighborhood $V(\subset U)$ of $o$
such that $\Xi'(V)=\mc L_s\cap V$
for some $s\in \{1,-1\}$.
Moreover,
\begin{enumerate}
\item if $m_{-s}$ is odd, then $\mc P(V)$ and $\mc N(V)$
are non-empty;
\item if $m_{-s}$ is even, then $\mc N(V)=V\setminus \mc L_s$
or $\mc P(V)=V\setminus \mc L_s$.
\end{enumerate}
\end{Thm}

\begin{proof}
Since $o$ is a quasi-umbilic, 
there exists $s\in \{1,-1\}$ satisfying
$m_s=0$ and $m_{-s}>0$ (cf.\ Lemma~\ref{lem:614}).
For the sake of simplicity, we only consider the
case where $s=1$, that is, $m_1=0$.
In this case, $0<m_{-1}< \infty$.
Then 
$$
N^2(\hat Q_f(z))=y^{m_{-1}}\psi_1(x)\psi_{-1}(y) \qquad
(y=\tfrac{u-v}{2}).
$$
If $m_{-1}$ is odd (resp.\ even),
then $N^2(\hat Q_f(z))$ changes (resp.\ does not change)
sign around $\mc L_{1}$,
since $\psi_1(0)\psi_{-1}(0)\ne 0$.
\end{proof}

In \eqref{eq:KO}, if we set $\omega=\hat \omega(z)dz$, 
then we can write
$$
g(z)=\epsilon_1 g_1(x)+\epsilon_{-1} g_2(y),\qquad
\hat \omega(z)=\epsilon_1 \hat \omega_1(x)
+\epsilon_{-1} \hat \omega_2(y).
$$
Since $dz=2(\epsilon_1 dx+\epsilon_{-1} dy)$, 
\eqref{eq:KO}
can be reduced to the following formula:
\begin{align}\label{eq:413}
f&=\sum_{i=1}^2 \int_{0}^{x_i} \Big( 
(-1)^{i+1}(1-g_i(x_i)^2),\, 2g_i(x_i),\,
1+g_i(x_i)^2
\Big)
\hat \omega_i(x_i)\,dx_i,
\end{align}
giving a time-like ZMC-surface, where $x_1:=x$ and $x_2:=y$.
In fact, the first fundamental form of $f$ is given by
$$
ds^2=-2(1-g_1(x)g_2(y))^2 \hat \omega_1(x)\hat \omega_2(y)\,dx\,dy,
$$
and 
$$
\tilde n:=\frac{1}{-1+g_1(x)g_2(y)}
\Big( -g_1(x)+g_2(y),\,1+g_1(x)g_2(y),\,g_1(x)+g_2(y)
\Big)
$$
is a unit normal vector field of $f$.
Then
$$
I\!I:=-2 \hat \omega_1(x)g'_1(x)\,dx^2-2 \hat \omega_2(y)g'_2(y)\,dy^2
$$
is the second fundamental form, and
\begin{equation}\label{eq:995}
W_f=\frac{1}{\Delta}\pmt{
0 & \hat \omega_2(y)g'_2(y) \\
\hat \omega_1(x)g'_1(x)& 0},
\end{equation}
is the Weingarten matrix of $f$,
where
$\Delta:=(-1+g_1(x)g_2(y))^2 \hat \omega_1(x)\hat \omega_2(y)$.

\begin{Example}
In the formula \eqref{eq:413}, if we set
$g_1:=x$, $g_2:=y^2$  (resp. $g_1:=x$, $g_2:=y^3$)
and $\hat \omega_1=\hat \omega_2=1$, 
then the resulting ZMC-surface $f_1$ (resp. $f_2$)
can be written as
\begin{align*}
&f_1(x,y)=
\Big(-\frac{x^3}{3}+x+\frac{y \left(y^4-5\right)}{5} ,
x^2+\frac{2 y^3}{3},\frac{x^3}{3}+x+\frac{y^5}{5}+y\Big),
\\
&\left(\text{resp.\,\,}
f_2(x,y)=
\Big(
-\frac{x^3}{3}+x+\frac{y \left(y^6-7\right)}{7},x^2
+\frac{y^4}{2},\frac{x^3}{3}
+x+\frac{y^7}{7}+y\Big)\right).
\end{align*}
Then, one can easily check that
 $D_{f_1}$ (resp. $D_{f_2}$)
is a positive scalar multiple of $y$  (resp. $y^2$),
and $o$ is not (resp. is) an admissible 
quasi-umbilic.
\end{Example}

\subsection*{Properties of umbilics}

We consider the case where $o$ is an umbilic.

\begin{Prop}\label{thm:D}
Let $f:U\to M^3_1(c)$ be a time-like CMC immersion.
Suppose that $m_1,m_{-1}<\infty$.
If $o$ is an umbilic, then 
there exists a neighborhood $V(\subset U)$ of $o$
such that
\begin{equation}\label{eq:966}
\Xi(V)=\{o\},\qquad \Xi'(V)
=\Big((\mc L_{1}\cup \mc L_{-1})\cap V\Big)\setminus \{o\}.
\end{equation}
Moreover, for each $s\in \{1,-1\}$,
$m_s$ is positive,
and the $($unique$)$ principal direction of $f$ along
$\mc L_s\setminus\{o\}$ is parallel to the line
$\mc L_s$. 
\end{Prop}

\begin{proof}
Since $o$ is an umbilic, we have $\hat Q_f(o)=0$.
In particular, $m_1,m_{-1}$ are positive.
By \eqref{eq:777}
and \eqref{eq:3-2},
we have 
\eqref{eq:966}.
Thus $\mc L_s\setminus \{o\}$ ($s\in \{1,-1\}$) consists of 
quasi-umbilics, and the last assertion
follows from Lemma~\ref{lem:614}.
\end{proof}

We now prove (I) in the introduction:

\begin{proof}[Proof of {\rm (I)}]
If $o$ is an umbilic, 
then $\Xi'(V)\cup \{o\}$ 
coincides with
$(\mc L_{1}\cup \mc L_{-1})\cap V$ 
by Proposition~\ref{thm:D}.
\end{proof}

\begin{Thm}\label{thm:1078}
Let $f:U\to M^3_1(c)$ be a time-like CMC immersion.
Suppose that  $m_1$ and $m_{-1}$ are finite and positive.
Then 
$o$ is an umbilic, and 
there exists a neighborhood $V$ of $o$
satisfying the following properties:
\begin{enumerate}
\item If $m_{1}$ or $m_{-1}$ is odd,
then $\mc P(V)$ and $\mc N(V)$ are both non-empty.
In particular, $o$ is not an
admissible umbilic $($cf.\ Definition~\ref{def:0}$)$.
\item 
Suppose that $m_{1}$ and $m_{-1}$ are both even.
Then $\psi_{1}(0)\psi_{-1}(0)$ 
is negative if and only if 
\begin{equation}\label{eq:Vstar}
V_*:=V\setminus (\mc L_1\cup \mc L_{-1})
\end{equation}
consists only of negative points.
On the other hand,
if $\psi_{1}(0)\psi_{-1}(0)$
is positive,
then there exists a pair of $C^\infty$ differentiable 
curvature line flows $(\mc F_1,\mc F_2)$ on $V$
satisfying the following properties:
\begin{enumerate}
\item 
If $f$ is real analytic, then 
$\mc F_i$ $(i=1,2)$
are real analytic, and
any real analytic curvature line flow of $f$ coincides with
$\mc F_1$ or $\mc F_2$. 
\item
When  $m_1/2$ and $m_{-1}/2$ are odd integers,
then the indices of the flows 
$\mc F_i$ $(i=1,2)$ at $o$ are $\pm 1$.
On the other hand, if
either $m_1/2$ or $m_{-1}/2$ is even,
then the indices of the flows 
$\mc F_i$ $(i=1,2)$ are equal to zero.
\end{enumerate}
\end{enumerate}
\end{Thm}

\begin{proof}
(1) can be proved easily.
So we assume that $m_1$ and $m_{-1}$ are both even
and positive.
Then 
$\psi_{1}(0)\psi_{-1}(0)<0$
holds
if and only if 
$N^2(\hat Q_f(z))$ is negative
on $V_*$ (cf. \eqref{eq:Vstar})
for a sufficiently small $V$.
In this case, 
$V_*$ consists 
only of negative points, proving 
the first part of (2).

Now we consider the case that $\psi_{1}(0)\psi_{-1}(0)>0$.
Then $N^2(\hat Q_f(z))$ is positive on $V_*$
for a sufficiently small $V$.
Since $m_1$ and $m_{-1}$ are both even,
we can write $m_1=2n_1$ and $m_{-1}=2n_{-1}$,
so that
\begin{equation}\label{eq:1111}
N^2(\hat Q_f)=\psi_1(x)\psi_{-1}(y) 
x^{2n_1}y^{2n_{-1}}\ge 0
\end{equation}
holds.
Hence we can write
$$
\hat Q_f=\delta \Big(\epsilon_1\alpha(x)x^{2n_1}
+\epsilon_{-1}\beta(y)y^{2n_{-1}}\Big)
\qquad (\delta \in \{1,-1\}),
$$
where $\alpha(0)$ and $\beta(0)$ are positive.
In particular, $\sqrt{\alpha(x)}$ and  $\sqrt{\beta(y)}$
can be regarded as smooth functions when $|x|$ and $|y|$
are sufficiently small.
Therefore,
(cf.~\eqref{eq:Af})
$$
A_f=\frac{\delta}{4} 
\pmt{\alpha(x)x^{2n_1}+\beta(y)y^{2n_{-1}}&
\alpha(x)x^{2n_1}-\beta(y)y^{2n_{-1}} \\
-\alpha(x)x^{2n_1}+\beta(y)y^{2n_{-1}}
& -\alpha(x)x^{2n_1}-\beta(y)y^{2n_{-1}}}
$$
and
\begin{align*}
&X_1:=\Big(x^{n_1}\sqrt{\alpha(x)}+y^{n_{-1}}\sqrt{\beta(y)}
\Big)\frac{\partial}{\partial u}+
\Big(-x^{n_1}\sqrt{\alpha(x)}+y^{n_{-1}}\sqrt{\beta(y)}\Big)
\frac{\partial}{\partial v}, \\
&X_2:=\Big(-x^{n_1}\sqrt{\alpha(x)}+y^{n_{-1}}\sqrt{\beta(y)}
\Big)\frac{\partial}{\partial u}
+ \Big(x^{n_1}\sqrt{\alpha(x)}+y^{n_{-1}}\sqrt{\beta(y)}\Big)
\frac{\partial}{\partial v}
\end{align*}
are eigenvector fields of $A_f$.
If $f$ is real analytic, then $(u,v)$ is a  real analytic
isothermal coordinate system,
so $\sqrt{\alpha(x)}$ and $\sqrt{\beta(y)}$ are real analytic functions,
and $X_i$ ($i=1,2$) are real analytic vector fields.

Since the absolute values of
the indices of $X_1$ and $X_2$ are the same,
we discuss the index of $X_1$.
Moreover, the index of $X_1$ does not depend on the 
choices of $\alpha$
and $\beta$, and hence we may set $\alpha=\beta=1$.
In this case, if $n_1$ (resp.~$n_{-1}$)
is even, $X_1$ is never positively proportional to $(-1,1)$
(resp.~$(-1,-1)$);
thus the index of $X_1$ at $o$ is equal to zero.
On the other hand,
when $n_{1}$ and $n_{-1}$ are both odd, one can easily
observe that
the index of $X_1$ at $p$ is equal to $1$.
\end{proof}

We now prove (II) in the introduction:

\begin{proof}[Proof of {\rm (II)}]
If $Q_f$ is non-degenerate at $o(:=p)$, then by \eqref{eq:614},
we can write $\hat Q_f=R(z) z^m$
and $N^2(R(o))\ne 0$.
In this case, 
we have $m=m_1=m_{-1}>0$, and
the sign of $N^2(R(o))$ coincides with that of
$\psi_1(0)\psi_{-1}(0)$.
Therefore, the assertions of (II) follow directly from
Theorem~\ref{thm:1078}.
\end{proof}

\begin{Exa}
We set $g:=z^2$ and $\omega:=dz$.
By
\eqref{eq:H666},
the Hopf differential is $Q_f=-2 zdz^2$, and so $m=1$ and $R(z)=-2$.
By \eqref{eq:KO}, we obtain a ZMC-immersion
\begin{align*}
f(u,v)
&=
\Big(\frac{-v\left(5 u^4+10 u^2 v^2+v^4-5\right)}{5},
\frac{2u\left(u^2+3v^2\right)}{3} , \\
&\phantom{aaaaaaaaaaaaaaaaaaaaaaaaaaaa}
\frac{u^5+10 u^3 v^2+5u v^4+5u}5\Big),
\end{align*}
such that $o$ 
is a non-admissible umbilic which is not
an accumulation point of umbilics.  
In fact, $D_f(u,v)(u^2-v^2)$ ($|u|\ne |v|$)
is positive on a sufficiently small neighborhood of $o$ in $\R^2$.
\end{Exa}

\begin{Exa}
We set $(g,\omega):=(z^3,dz)$ and substitute this into
\eqref{eq:KO}. Then 
the  ZMC-immersion
\begin{align*}
f(u,v)
&=
\Big(\frac{-v\left(7 u^6+35 u^4 v^2+21 u^2 v^4+v^6-7\right)}{7},\\
&\phantom{aaaaaaa}\frac{u^4+6 u^2 v^2+v^4}{2},
\frac{u\left({u^6}+21 u^4 v^2+35 u^2 v^4+7v^6+7\right)}7\Big)
\end{align*}
satisfying
$Q_f=-3z^2dz^2$ (cf.~\eqref{eq:H666})
is obtained.
In particular, we have $m=2$, and $o$ is an 
isolated umbilic.
A real analytic vector field $V:=-v\frac{\partial}{\partial u}+u \frac{\partial}{\partial v}$
points in the principal directions of 
$f$, having index $1$ at the origin $o$.
\end{Exa}

\begin{Exa}
We set $(g,\omega)=(z^5,dz)$ and substitute this into
\eqref{eq:KO}. Then 
the  ZMC-immersion
\begin{align*}
f(u,v)
&=
\Big(-u^{10} v-15 u^8 v^3-42 u^6 v^5-30 u^4 v^7-5 u^2 v^9-\frac{v^{11}}{11}+v, \\
&\phantom{aaaaa}
\frac{1}{3} \left(u^2+v^2\right) \left(u^4+14 u^2 v^2+v^4\right), \\
&\phantom{aaaaaaaaaa}
\frac{u^{11}}{11}+5 u^9 v^2+30 u^7 v^4+42 u^5 v^6+15 u^3 v^8+u v^{10}+u\Big)
\end{align*}
having $Q_f=-5z^4dz^2$ 
is obtained.
In particular, $m=4$ holds, and $o$ is 
an isolated umbilic.
In this setting,
$X:=(u^2+v^2)\frac{\partial}{\partial u}-2uv\frac{\partial}{\partial v}$ 
is a real analytic vector field 
pointing in the principal directions of $f$
such that the index of $X$ at the origin $o$ is equal to $0$.
\end{Exa}

Finally, we give an example such that $Q_f$ is degenerate:

\begin{Exa}
In the formula \eqref{eq:413}, if we set
$g_1:=x^3$, $g_2:=y^7$
and $\hat \omega_1=\hat \omega_2=1$, then the resulting ZMC-surface $f$ 
has $m_1=2$ and $m_{-1}=6$, which 
can be written as
$$
f(x,y)=
\left(-\frac{x^7}{7}+x+\frac{y^{15}}{15}-y,\,\frac{2 x^4+y^8}{4},\,\frac{x^7}{7}+x+\frac{y^{15}}{15}+y\right).
$$
This surface has an isolated umbilic 
at the origin. By \eqref{eq:995},
the vector fields 
$\mp\sqrt{7}y^3\frac{\partial}{\partial x}
+\sqrt{3}x\frac{\partial}{\partial y}$ 
point in the principal directions of $f$,
and their indices at the origin are equal to $\pm 1$.
\end{Exa}

\begin{acknowledgements}
The authors thank 
Shintaro Akamine and Wayne Rossman for their valuable comments.
\end{acknowledgements}

\appendix
\section{Space-like CMC surfaces}

In this appendix, 
we show that the local behavior of curvature line flows of
space-like CMC surfaces 
in $M^3_1(c)$ ($c\in \R$)
is quite similar to that of CMC surfaces in the Euclidean space.

Let $f:U\to M^3_1(c)$ ($c\in \R$) be
a space-like CMC surfaces defined on a domain $U$
in the $uv$-plane $\R^2$.
For each $p\in U$,
there exists an isothermal coordinate 
system centered at $p$.
Hence we may assume that the coordinate system $(u,v)$ is
an isothermal coordinate system on $U$.
Then the first fundamental form of $f$
can be expressed as
$
ds^2=e^{2\sigma}(du^2+dv^2),
$
where $\sigma$ is a smooth function on $U$.
We first consider the case $c=0$, that is,
$M^3_1(c):=\R^3_1$.
Let $\nu$ be a unit normal vector field of $f$ defined on $U$.
If we set 
$
\bm F=(f_u,f_v,\nu),
$
then this frame field satisfies the following Frenet-type equations:
$\bm F_u=\bm F\Omega$ and $\bm F_v=\bm F\Lambda$,
where
$$
\Omega=\pmt{
\sigma_u & \sigma_v & -e^{-2\sigma}L \\
-\sigma_v & \sigma_u & -e^{-2\sigma}M \\
-L               &   -M             & 0  
},
\qquad
\Lambda=\pmt{
\sigma_v & -\sigma_u & -e^{-2\sigma}M \\
\sigma_u & \sigma_v & -e^{-2\sigma}N \\
-M               &   -N             & 0  
}.
$$
The integrability condition of this system is
$
\Lambda \Omega+\Omega_v=\Omega\Lambda +\Lambda_u.
$
Computing this, we obtain the following three relations:
\begin{align}\label{eq:Gauss2}
&LN-M^2=
e^{-2 \sigma}(\sigma_{uu}+\sigma_{vv}), \\[4pt]
&L_v-M_u=\sigma_v (L+N), \quad \label{eq:Cod2}
N_u-M_v=\sigma_u (L+N).
\end{align}
The first one
is the Gauss equation, and the others are the Codazzi equations.
Similarly, if $f$ is a map into the de Sitter 3-space in $\R^4_1$
or the anti-de Sitter space in $\R^4_2$, then 
by similar computations,
we can show that the Codazzi equations are
written as in \eqref{eq:Cod2}.
The mean curvature function of $f$ is given by
$
H:=(1/2)e^{-2\sigma}(L+N).
$
We now assume that $H$ is constant.
Then, as in the case of surfaces in the Euclidean 3-space $\E^3$,
we obtain the following:

\begin{Prop}
The function $\hat Q_f:=(L-N)-2i M$ is holomorphic
as a function of $z:=u+i v$, where
$i:=\sqrt{-1}$. 
\end{Prop}

We call $Q_f:=\hat Q_f dz^2$ the {\it Hopf differential} of $f$.
Imitating the case of CMC surfaces in
 $\E^3$ (cf.~Hopf~\cite{H}),
the following assertion can be shown:

\begin{Fact}\label{f1144}
Let  $f:U\to M^3_1(c)$ be a space-like 
CMC surface with an umbilic $p$ that is not totally umbilic.
Then $p$ is an isolated umbilic of $f$, and the index of $p$ is $-m/2$,
where $m(\ge 1)$ is the order of the 
zero of the Hopf differential of $f$ at $p$.
\end{Fact}

Conversely, the following assertion holds:

\begin{Prop}
For each positive integer $m$,
there exists a space-like ZMC-immersion 
with an isolated umbilic at which the Hopf differential has a zero
of order $m$.
\end{Prop}

Indeed,
if we set 
$g:=-z^{m+1}/(m+1)$ and $\omega:=dz$
on the complex plane $(\C,z)$,
where $m$ is a positive integer and $z=u+iv$, then 
the representation formula (cf.~Kobayashi \cite{OK})
$
f(u,v):=\op{Re}\int_o^z(1+g^2,i(1-g^2),2g)\omega
$
gives
a space-like ZMC-immersion 
whose Hopf differential $Q_f$ is $-\omega dg=z^m dz^2$. 
By Fact~\ref{f1144}, the index of the curvature line flows
at $o$ is $-m/2$.

\section{Properties of smooth para-holomorphic functions}

We set
\begin{equation} \label{eq:e1e2}
\epsilon_1:=\frac{1+j}2,\qquad \epsilon_{-1}:=\frac{1-j}2.
\end{equation}
Then, clearly,
\begin{equation}\label{eq:eee}
\epsilon_1 \epsilon_{-1}=0, \qquad
\overline{\epsilon_s}=\epsilon_{-s},\qquad
\epsilon_s^2=\epsilon_s
\qquad (s\in \{1,-1\}).
\end{equation}
For $z\in \check \C$, there exists a unique
real number
$\Pi_s(z)$ for $s\in \{1,-1\}$ such that
$$
z=\Pi_1(z)\epsilon_1+\Pi_{-1}(z)\epsilon_{-1},\quad
\Pi_s(z)=u+s v \qquad (u:=\op{Re}(z),\,\, v:=\op{Im}(z)).
$$
For each $s\in \{1,-1\}$, we have
\begin{align}
&
\Pi_s(z+w)=\Pi_s(z)+\Pi_s(w), \quad
 \Pi_s(c z)=c \Pi_s(z)\quad
(c\in \R,\,\,z,w\in \check \C).
\end{align}
Moreover, for $z\in \check \C$, we have
\begin{align}
&  \label{eq:c0}
\Pi_s(\bar z)=\Pi_{-s}(z),\qquad
 \Pi_s\circ \Pi_s(z)=\Pi_s(z), \\[2mm]
& N^2(z)
=u^2-v^2=\Pi_1(z) \Pi_{-1}(z). \label{eq:c0b}
\end{align}

Let $A(u,v)$ and $B(u,v)$ be real-valued smooth functions
defined on a domain $U$ in $\check \C$.
We consider the function $h(z)$ ($z=u+jv$)
defined by $h(z)=A(u,v)+j B(u,v)$.
If we set
$$
h_z:=\frac12 (h_u+j h_v),\qquad
h_{\bar z}:=\frac12 (h_u-j h_v),
$$
then $h_{\bar z}$ vanishes identically if and only if $h$ is
para-holomorphic,
that is,
$A$ and $B$ satisfy the para-Cauchy--Riemann equations
$A_u=B_v$ and $A_v=B_u$.
In this case, the derivative of $h$
is defined as a para-holomorphic function 
\begin{equation}\label{eq:P}
\frac{dh}{dz}:=X_u+j Y_u~(=h_z),
\end{equation}
where $X:=\op{Re}(h)$ and $Y:=\op{Im}(h)$.
If $\phi_i$ ($i=1,2$) are smooth functions defined 
on open intervals $I_i(\subset \R)$, respectively,
then
\begin{align}\label{X}
\phi_1\vee \phi_2(u,v):=
\frac{\phi_1(u+v)+\phi_2(u-v)}2+j\left
(\frac{\phi_1(u+v)-\phi_2(u-v)}2\right)
\\ \nonumber
\phantom{aaaaaaaaaaaaaaaaaaaaaaaa} (u+v\in I_1,\,\, u-v\in I_2)
\end{align}
is a para-holomorphic function.
Conversely,
the following fact is well-known: 

\begin{Fact}\label{fact:A1}
Let $h$ be a para-holomorphic function on $U$.
Then, for each $p\in U$, there exist $\epsilon>0$
and two smooth functions $\phi_i:(-\epsilon,\epsilon)\to \R$
$(i=1,2)$ satisfying $\phi_i(0)=0$, such that 
$h(u,v)=h(p)+\phi_1 \vee \phi_2(u,v)$
holds for all $u+v,u-v\in (-\epsilon,\epsilon)$.
\end{Fact}

\begin{Prop}\label{prop:c}
Let $h(z)$ be a smooth para-holomorphic function defined on
a neighborhood $U(\subset \check \C)$ of the origin $o$
satisfying $h(o)=0$.
If we set
\begin{equation}\label{eq:p1p2}
\phi_1(z):=\Pi_1(h(z)),\qquad \phi_{-1}(z):=\Pi_{-1}(h(z)),
\end{equation}
then the following assertions hold:
\begin{enumerate}
\item 
$\phi_1$ $($resp.\ $\phi_{-1})$
is a function of $x:=(u+v)/2$
$($resp.\ $y:=(u-v)/2)$,
\item $h(z)=\epsilon_{1}\phi_1(x)+\epsilon_{-1}\phi_{-1}(y)$~holds,
\item $N^2(h(z))=\phi_1(u+v)\phi_{-1}(u-v)$~holds on $U$.
\end{enumerate}
\end{Prop}

\begin{proof}
(1) follows from Fact \ref{fact:A1}~immediately.
(2) and (3) can be proved by using formulas in the first part of this
appendix.
\end{proof}

\end{document}